\documentclass[11pt, a4paper]{amsart}
\usepackage[english]{babel}
\usepackage{amsfonts,amsmath}	
\usepackage{epsfig}
\usepackage{graphicx}		
\usepackage{amsthm,amssymb}
\usepackage[subnum]{cases}
\usepackage{url}
\usepackage{verbatim}
\usepackage{texdraw}
\usepackage{cite}
\makeatletter
\def\url@leostyle{%
  \@ifundefined{selectfont}{\def\UrlFont{\sf}}{\def\UrlFont{\small\ttfamily}}}
\makeatother
\input xy
\xyoption{all}				

\theoremstyle{definition}
\newtheorem{thm}{Theorem}

\newtheorem{lem}[thm]{Lemma}
\newtheorem{defi}[thm]{Definition}
\newtheorem{prop}[thm]{Proposition}

\newtheorem{rem}[thm]{Remark}
\title{Theta characteristics of hyperelliptic graphs}
\author{Marta Panizzut}
\email{{\tt marta.panizzut@wis.kuleuven.be}}
\address{KU Leuven, Department of Mathematics, Celestijnenlaan 200B, 3001 Heverlee, Belgium}
\thanks{The author is supported by the Research Project G.0939.13N of the Research Foundation - Flanders (FWO)}

\begin{document}

\begin{abstract} We study theta characteristics of hyperelliptic metric graphs of genus $g$ with no bridge edges. These graphs have a harmonic morphism of degree two to a metric tree that can be  lifted to morphism of degree two of a hyperelliptic curve~$X$ over $K$ to the projective line, with $K$ an algebraically closed field of char$(K) \not =2$, complete with respect to a  non-Archimedean valuation, with residue field $k$ of char$(k)\not=2$. The hyperelliptic curve has $2^{2g}$ theta characteristics. We show that for each effective theta characteristics on the graph, $2^{g-1}$~even and $2^{g-1}$~odd theta characteristics on the curve specialize to it; and $2^g$~even theta characteristics on the curve specialize to the unique not effective theta characteristics on the graph.
\end{abstract}
\maketitle

\begin{section}{Introduction}
A theta characteristic of a smooth algebraic curve $X$ of genus~$g$ over an algebraically closed field $K$ of char$(K) \not =2$, is an element $\mathcal{L}$ of $\textrm{Pic}^{g-1}(X)$ such that $\mathcal{L}\otimes \mathcal{L} = \omega_X$. They are called odd or even according to the parity of $\textrm{h}^0(X, \mathcal{L})$. There are $2^{2g}$ theta characteristics, $2^{g-1}(2^g-1)$ of which are odd, and the parity of a theta characteristic does not change if $X$ varies in an algebraic family, see \cite{Mum}.

Theta characteristics on metric graphs have been introduced by Zharkov in \cite{Zha}. Given a metric graph $\Gamma$ of genus $g$, a {\sl theta characteristic} is a divisor class $[D]$ in $\textrm{Pic}^{g-1}(\Gamma)$ such that $2[D] = [K_{\Gamma}]$. Zharkov showed that there are $2^g$ theta characteristics and they are all but one effective, meaning that $D$ is linearly equivalent to an effective divisor. 

Chan and Jiradilok \cite{CJ} consider $K_4$-curves, namely smooth, proper curves of genus $3$ over a non-Archimedean field, whose Berkovich skeleton $\Gamma$ is a complete metric graph $K_4$ on four vertices. Such curves have $28$ bitangents to a canonical embedding. Each of these bitangent corresponds to an effective theta characteristic given by the sum of the intersection points. The metric graph $\Gamma$ has $7$ effective theta characteristics. In their article the authors show that the $28$ bitangents specialize to the effective theta characteristics in seven groups of four, answering, in the case of $K_4$-curve, a question posed by Baker, Len, Morrison, Pflueger and Ren in \cite{BLMPR}.

\vspace{\baselineskip}
Our starting point is a working hypothesis formulated by Marc Coppens and motivated by the results of Chan and Jiradilok. Let $\Gamma$ be the dual graph of a strongly semistable model of a smooth curve $X$ over an algebraically closed field $K$ of char$(K) \not =2$, complete with respect to a non-trivial, non-Archimedean valuation. For each effective theta characteristic on the graph, $2^{g-1}$ even and $2^{g-1}$ odd theta characteristics on the curve specialize to it; and $2^g$ even theta characteristics on the curve specialize to the unique non-effective theta characteristics on the graph.

We focus on theta characteristics on hyperelliptic graphs and curves.  We add the hypothesis that the residue field $k$ of $K$ is of char$(k) \not=2$. A hyperelliptic metric graph~$\Gamma$ of genus $g$ has a harmonic morphism of degree two to a metric tree. Under the hypothesis that all the bridge edges are contracted, this morphism has $g+1$ {\sl ramification points}, see Definition \ref{rp}. The harmonic morphism can always be  lifted to morphism of degree two of a hyperelliptic curve~$X$ over $K$ of genus $g$ to the projective line. This covering has $2g+2$ ramification points. The main result we will prove is Theorem~\ref{counting}, that states that the theta characteristics on the curve $X$ specialize to theta characteristics on the graph $\Gamma$ as described in the working hypothesis. 

In order to prove this result, we study the specialization of the ramification points on the curve $X$ to the metric graph~$\Gamma$, showing in Theorem \ref{ramification} that they specialize in pairs to the ramification points on~$\Gamma$. In Section \ref{thetacharacteristics} we prove that theta characteristics on a hyperelliptic graphs  can be described using the hyperelliptic series $g^1_2$ and the ramification points, as it is done for theta characteristics on hyperelliptic curves.

Combining the study of the specialization of ramifications points and these descriptions of theta characteristics, it is possible to count the number of even and odd theta characteristics on the curve that specialize to a given one on the graph. These computations are explained in Section \ref{5} and lead to the proof of Theorem \ref{counting}.

\end{section}
\begin{subsection}*{Acknowledgments} I am deeply grateful to Marc Coppens for sharing his working hypothesis with me, for his guidance and valuable comments on earlier drafts of this paper. I am also very grateful to the reviewer for many valuable remarks and in particular for pointing out my mistakes in a previous version of the proof of Theorem \ref{counting}.   I warmly thank Filip Cools and Farbod Shokrieh for fruitful discussions and many suggestions which improved this manuscript. 
\end{subsection}

\begin{section}{Preliminaries} \label{preliminaries}

\begin{subsection}{Metric graphs}
A {\sl topological graph $\Gamma$} is a compact, connected topological space such that for every $p \in \Gamma$ there exists a neighborhood $U_p$ of $p$ homeomorphic to a union of segments in $\mathbb{R}^2$ connecting the origin $(0,0)$ with $r$ points, every two of which lie in different lines through the origin. The number $r$ is called {\sl valence} and indicated with val$(p)$. Only for a finite number of points val$(p) \not =2$. These points are called {\sl essential vertices}. 
A {\sl metric graph} is a topological graph equipped with a complete inner metric on $\Gamma \setminus V_{\infty}(\Gamma)$, where $V_{\infty}(\Gamma)$ is a subset of vertices of valence one called {\sl infinite vertices}. The metric can also be extended to the infinite vertices, stating that their distance from all the other points of the graph is infinite.  

A vertex set $V(\Gamma)$ is a finite subset of points of $\Gamma$ containing the essential vertices. The closures of the connected components of $\Gamma \setminus V(\Gamma)$ are called {\sl edges}. The set $E(\Gamma)$ consists of the edges associated to $V(\Gamma)$. The {\sl infinite edges} are the edges adjacent to infinite vertices. A {\sl bridge edge} is an edge $e$ such that $\Gamma \setminus \mathring{e}$ is disconnected. From the metric it is possible to associate to every edge $e \in E(\Gamma)$  a length $l(e) \in \mathbb{R} \cup \{ \infty\}$. We have $l(e) = \infty$ if and only if $e$ is an infinite edge. 

Given a vertex set of a metric graph $\Gamma$, its genus $g$ is defined as the first Betti number, $g = |E(\Gamma)| - |V(\Gamma)| +1$. 

The set of {\sl tangent directions $T_p(\Gamma)$} at $p$ is the set of connected components of $U_p \setminus \{p\}$, where $U_p$ is a neighborhood of $p$ as defined above. If using another neighborhood $U_p'$, we identify the components where the intersection of $U_p \setminus \{p\}$ and $U_p' \setminus \{p\}$ is not empty.  

\vspace{\baselineskip}
A {\sl divisor} $D$ on a metric graph $\Gamma$ it is an element of the free abelian group $\textrm{Div}(\Gamma)$ on the points of the graph,  
\[
D = \sum_{p \in \Gamma} a_p \, (p), \ \ \textrm{with} \ a_p \in \mathbb{Z}. 
\]
The {\sl degree} of $D$ is the sum of its coefficients $\textrm{deg}(D) = \sum_{p \in \Gamma} a_p$. A divisor is {\sl effective} if $a_p \geq 0$ for every $p \in \Gamma$. The {\sl support $\textrm{supp}(D)$ of a divisor $D$} is the set of points of $\Gamma$ such that $a_p \not =0$. 

The {\sl canonical divisor} $K_{\Gamma}$ has at every point $p$ coefficient $\textrm{val}(p)-2$.

A {\sl rational function} $f$ on $\Gamma$ is a continuous piecewise linear function $f: \Gamma \rightarrow \mathbb{R}$ with integer slopes and only finitely many pieces. The {\sl principal divisor} $\textrm{div}(f)$ associated to $f$ is the divisor whose coefficient at $p$ is given by the sum of the outgoing slopes of $f$ at that point. 

Two divisor $D_1$, $D_2 \in \textrm{Div}(\Gamma)$ are {\sl linearly equivalent}, $D_1 \sim D_2$, if there exists a rational function $f$ such that $D_1  - D_2 = \textrm{div}(f)$. The group of divisor classes of degree $g-1$ is indicated with $\textrm{Pic}^{g-1}(\Gamma)$. The {\sl linear system} $|D|$ is the set of the effective divisor linearly equivalent to $D$. 

The {\sl rank $\textrm{rk}_{\Gamma}(D)$ of a divisor} is defined as $-1$ if $D$ is not equivalent to any effective divisor, otherwise
\[
\textrm{rk}_{\Gamma}(D)= \max \{ r \in \mathbb{Z}_{\geq 0}| \ |D-E| \not = \emptyset \ \ \  \forall \ E \in \textrm{Div}(\Gamma), \ E \geq 0, \ \textrm{deg}(E)=r\}. 
\]

Let $\Gamma$ be a metric graph and $X$ be a closed connected subset of $\Gamma$. Given $p \in \partial X$, the {\sl outgoing degree of $X$ at $p$} is defined as the number of tangent directions leaving $X$ at $p$. Let $D$ be a divisor on $\Gamma$. A boundary point $p \in \partial X$ is {\sl saturated with respect to $X$ and $D$} if $D(p) \geq \textrm{outdeg}_X(p)$, and {\sl non-saturated} otherwise. Let $q$ be a point of $\Gamma$. 
A divisor $D$ is {\sl $q$-reduced} if it is effective in $\Gamma \setminus \{q\}$ and each closed connected subset $X$ of $ \Gamma \setminus \{ q \}$ contains a non-saturated boundary point.

\begin{thm}[Proposition 7.5 in \cite{MZ}]
Let $D$ be a divisor on a metric graph $\Gamma$. For every point $p \in \Gamma$ there exists a unique $p$-reduced divisor linearly equivalent to $D$. 
\end{thm}




\end{subsection}

\begin{subsection}{Harmonic morphisms} 
A morphism $\varphi: \Gamma' \rightarrow \Gamma$ between metric graphs is a continuous map such that for certain vertex and edge sets of $\Gamma$ and $\Gamma'$ it holds that $\varphi(V(\Gamma')) \subseteq V(\Gamma)$, for every edge $e$ of $\Gamma$ it holds that $\varphi^{-1}(e)$ consists of edges of $\Gamma'$, and $\varphi$ restricted to an edge $e' \in E(\Gamma')$ is a dilation by some factor $d_{e'}(\varphi) \in \mathbb{Z}_{\geq 0}$. A morphism is {\sl finite} if $d_{e'}(\varphi) \not =0$ for every edge $e'$. Given $p' \in \Gamma'$, $v' \in T_{p'}(\Gamma')$ a tangent direction and $e' \in E(\Gamma')$ the edge in the direction of $v'$, the  directional derivative $d_{v'}(\varphi)$ of $\varphi$ in the direction $v'$ is defined as $d_{v'}(\varphi) := d_{e'}(\varphi)$. 

Let $p=\varphi(p')$. The morphism $\varphi$ is {\sl harmonic at $p'$} if the integer
\[
d_{p'}(\varphi) = \sum_{\substack{v' \in T_{p'}(\Gamma), \\  \varphi(v') = v}} d_{v'}(\varphi)
\]
is independent of the choice of a tangent direction $v \in T_p(\Gamma)$. The number $d_{p'}(\varphi)$ is called the {\sl degree of $\varphi$ at $p'$}. 

If the morphism $\varphi$ is surjective and harmonic at every $p' \in \Gamma'$, we say that it is {\sl harmonic}. The {\sl degree of $\varphi$} is define as
\[
\textrm{deg}(\varphi) = \sum_{\substack{p' \in \Gamma', \\ \varphi(p') = p}} d_{p'}(\varphi) 
\]
and it does not depend on $p \in \Gamma$.

Let $\varphi: \Gamma' \rightarrow \Gamma$ be a harmonic morphism of metric graphs. The {\sl ramification divisor} of $\varphi$ is the divisor $R = \sum_{p \in \Gamma'} R(p')(p')$ whose coefficient at the point $p' \in \Gamma$ is given by 
\[
R(p') = 2d_{p'}(\varphi)  - 2 - \sum_{v' \in T_{p'}(\Gamma')} \Big( d_{v'}(\varphi) -1 \Big).
\]
\end{subsection}

\begin{subsection}{Specialization} 
Linear systems on curves specialize to linear systems on graphs as explained in \cite{Bak}, see also \cite{AB}. Let $K$ be a algebraically closed field, complete with respect to a non-trivial, non-Archimedean valuation. Let $R$ be the valuation ring and $k$ the residue field. Let $X$ be a smooth, proper, connected curve over $K$ and $X^\textrm{an}$ its analytification. Given a strongly semistable model $\mathfrak{X}$ of $X$ over $R$ such that the special fiber has only irreducible components of genus zero, we can construct the dual graph of the special fiber. Inside $X^{\textrm{an}}$ there is a natural metric graph $\Gamma$. The associated finite graph is the dual graph of the special fiber. There is a retraction map $\tau: X^{\textrm{an}} \rightarrow \Gamma$ that extends by linearity to a homomorphism $\tau_{*}: \textrm{Div}(X) \rightarrow \textrm{Div}(\Gamma)$. The rank of the divisor can only increase: 
\begin{thm}[Specialization Lemma, Corollary 2.10 in \cite{Bak}] Let $D$ be a divisor on $X$. The following inequality holds:
\[
\textrm{rk}_{\Gamma}(\tau_{*}(D)) \geq  \textrm{rk}_{X}(D).
\]
\end{thm}

\end{subsection}

\end{section}

\begin{section}{Hyperelliptic graphs and curves}\label{specialization}
\begin{subsection}{Hyperelliptic graphs}
In this section we give the definition of hyperelliptic graphs and their main properties following \cite{ABBR2}. In the article the authors work with augmented metric graph: metric graphs with a genus function $g: \Gamma \rightarrow \mathbb{Z}_{\geq 0}$; in this paper we  assume that all the graphs are {\sl totally degenerated}, so the genus function is identically zero. We recall lifting results of hyperelliptic graphs to hyperelliptic curves and we study the specialization of the ramification points. More references are \cite{BN2}, \cite{Cha} and \cite{KY1}.
\begin{defi} A metric graph $\Gamma$ of genus at least $2$ is {\sl hyperelliptic} if it has a $g^1_2$, that is  a linear system of degree $2$ and rank $1$. 
\end{defi} 
The following remark is an easy generalization of Proposition 5.5 in \cite{BN2}, see also \cite{Fac}. 
\begin{rem}
If there exists a $g^1_2$ on an augmented metric graph, then it is unique. 
\end{rem}
The property of a graph being hyperelliptic can also be described using harmonic morphisms, Proposition 4.12 of \cite{ABBR2}, see also \cite{BN2} and \cite{Cha}. \begin{defi} A {\sl minimal graph} is a metric graph with no $1$-valent vertices. 
\end{defi}
 
\begin{thm}\label{hyper} Let $\Gamma$ be a minimal metric graph of genus at least $2$. Then the following are equivalent: 
\begin{enumerate}
\item $\Gamma$ is hyperelliptic;
\item there exists an involution $\iota$ on $\Gamma$ such that $\Gamma / \iota$ is a metric tree;
\item there exists a finite harmonic morphism of degree two to a metric tree $T$, $\varphi: \Gamma \rightarrow T$.
\end{enumerate} 
Furthermore, the hyperelliptic involution is unique. 
\end{thm} 

The property of being hyperelliptic is compatible with the contraction of bridge edges, as shown in Corollary 5.11 of \cite{BN2} and Corollary 3.12 of \cite{Cha}:
\begin{lem} Let $\Gamma$ be a metric graph. Suppose that $\Gamma$ has a bridge and $\Gamma'$ is the graph obtained by contracting the bridge. Then $\Gamma$ is hyperelliptic if and only if $\Gamma'$ is hyperelliptic. 
\end{lem}
From now on we assume that all the bridge edges are contracted. We consider a minimal hyperelliptic graph $\Gamma$ and we indicate with $\varphi$ the finite harmonic morphism of degree two to a metric tree, $\varphi: \Gamma \rightarrow T$. We study its ramification divisor $R$. 

For every direction $v \in T_p(\Gamma)$ we can have $d_v(\varphi) = 1$ or $d_v(\varphi) =2$ since the morphism is finite and of degree two. We are supposing that the bridges are contracted, therefore it cannot exist a direction $v$ for which $d_v(\varphi) = 2$. The coefficient of $R$ at $p$ is  
\[ R(p) = 2d_p(\varphi) -2 -  \sum_{v \in T_p(\Gamma)}\Big( d_v(\varphi)-1 \Big) = 2 d_p(\varphi) - 2.
\] 
It follows that $R(p) = 2$ for $d_p(\varphi)=2$, and  $R(p) = 0$ for $d_p(\varphi)=1$.

\begin{defi} \label{rp}
 A point of a hyperelliptic graph at which the ramification divisor has coefficient two is called a {\sl ramification point}.
\end{defi}
See Figure \ref{tcfig} for a graphical representation of ramification points. 
\begin{figure}[h]
\centering
\includegraphics[width=2cm,height=3cm]{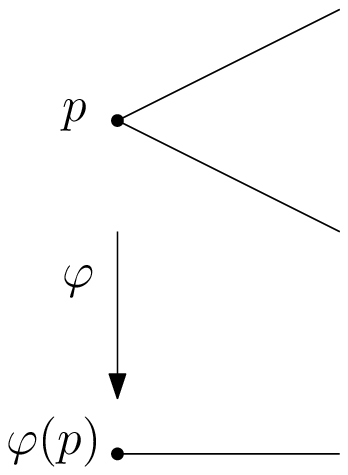} 
\qquad \qquad 
\includegraphics[width=2.5cm,height=4cm]{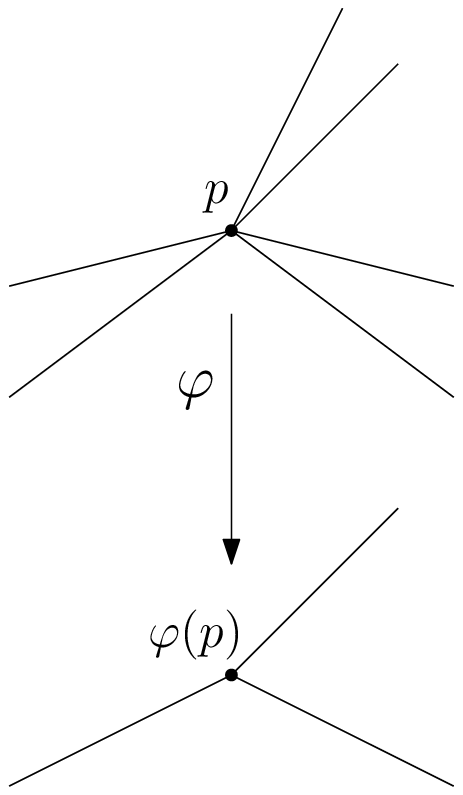} 
\label{tcfig}
\caption{The harmonic moprhism $\varphi$ at a ramification point~$p$ with val$(\varphi(p))=1$ and val$(\varphi(p))=3$.}
\end{figure}
\begin{prop}
Let $\Gamma$ be a minimal hyperelliptic graph and $\varphi: \Gamma \rightarrow T$ its finite harmonic morphism of degree two to a metric tree. Then $\Gamma$ has $g+1$ ramification points. 
\end{prop}
\begin{proof}
It follows from the graph-theoretic analogue of the Riemann-Hurwitz formula (see Theorem 2.9 of \cite{BN2} and Section 2.14 of \cite{ABBR1}) that the degree of the ramification divisor is $2g+2$. Therefore there are $g+1$ points for which the coefficient is two. 
\end{proof}

\begin{prop}
Let $r_i$ be a ramification point of a hyperelliptic graph $\Gamma$. Then $2(r_i) = g^1_2$.
\end{prop}
\begin{proof}
The ramification points $r_i$ of $\Gamma$ are such that rk$(2(r_i)) = 1$. In fact, each divisor $2(r_i)$ is the pull-back of the image point of $r_i$ in the metric tree: by Proposition 4.2 in \cite{ABBR2} the rank is at least $1$. Moreover the rank cannot be larger because this would imply that the genus of $\Gamma$ is zero. 

\end{proof}

\end{subsection}

\begin{subsection}{Lifting of hyperelliptic graphs}
Let $K$ be a algebraically closed field, complete with respect to a non-trivial, non-Archimedean valuation, such that char$(K) \not =2$ and char$(k) \not= 2$, with $k$ the residue field. 
Let $\Gamma$ be a minimal hyperelliptic metric graph, so there exists a finite harmonic morphism $\varphi: \Gamma \rightarrow T$ to a tree of degree two. In what follows we will briefly explain how this morphism can be lifted to a finite morphism $\varphi: X \rightarrow \mathbb{P}^1$ of degree two obtaining a hyperelliptic curve $X$. The morphism has $2g+2$ ramification points $R_1, R_2, \dots , R_{2g+2}$. We will study the specialization of these ramification points to the metric graph $\Gamma$. For the definition of lifting of a morphism and more results in this context we refer to \cite{ABBR1} and \cite{ABBR2}.
\begin{thm} \label{ramification} Let $\Gamma$ be a minimal hyperelliptic graph and $\varphi: \Gamma \rightarrow T$ a finite harmonic morphism of degree two to a tree. This morphism can be lifted to a finite morphism $\varphi: X \rightarrow \mathbb{P}^1$ lifting it. The $2g+2$ ramification points on $X$ specialize in pairs to the $g+1$ ramification points on $\Gamma$. 
\end{thm}
\begin{proof} By Section 3 of \cite{ABBR2} we know that if $p$ is not a ramification point, locally at this point the lift is given by the identity map $\varphi: \mathbb{P}^1 \rightarrow \mathbb{P}^1$. Instead if $p$ is a ramification point the local model is a map $\varphi: \mathbb{P}^1 \rightarrow \mathbb{P}^1$ of degree $2$. This morphism has two ramification and two branch points. By results in \cite{ABBR1} and \cite{ABBR2}, we modify the morphism $\varphi: \Gamma \rightarrow T$ by adding two infinite edges at the point $r$ that we map to two infinite edges added at $\varphi(r)$ with degree along these edges equal to two. The morphism $\varphi: \Gamma \rightarrow T$ can be lifted to a finite morphism of degree two $\varphi: X \rightarrow \mathbb{P}^1$, such that the ramification points $R_i$ on $X$ are the points that are mapped to the infinite vertices coming from the modification of $\Gamma$. The infinite vertices retract to the ramification points~$r_i$ on $\Gamma$, therefore, the ramification points on $X$ specialize in pairs to the ones on $\Gamma$.
\end{proof}
\end{subsection}
\end{section}

\begin{section}{Theta characteristics} \label{thetacharacteristics}
Now we focus on theta characteristics of hyperelliptic curves and graphs. 

As we said before any hyperelliptic curve $X$ of genus $g$ has a $g^1_2 = |D|$ and a degree two morphism $\varphi: X \rightarrow \mathbb{P}^1$ with $2g+2$ ramification points $R_1, R_2, \dots , R_{2g+2}$. 

Every theta characteristic on $X$ is of the form $\mathcal{L} = \mathcal{O}_X(E)$ where
\[E = m D+ R_{i_1} + \cdots + R_{i_{g-1-2m}},\]
with $-1\leq m \leq \frac{g-1}{2}$ and the points $R_{i_j}$ are distinct. 

This representation is unique for $m\geq 0$, while if $m =-1$ there is the following relation 
\begin{equation} \label{equivalenza} -D + R_{i_1} + \cdots R_{i_{g+1}} \sim -D + R_{j_1} + \cdots + R_{j_{g+1}}, \end{equation}
if $\{i_1, \dots, i_{g+1}, j_1, \dots, j_{g+1}\} = \{ 1, 2, \dots, 2g+2\}$. 

Moreover $h^0(X, \mathcal{L})=m+1$, see \cite{Mum}, \cite{Mum2} and \cite{ACGH}.

\bigskip
Let $\Gamma$ be a metric graph of genus $g$ and $K_{\Gamma}$ its canonical divisor. 
\begin{defi} A divisor class $[F] \in \textrm{Pic}^{g-1}(\Gamma)$ is called {\sl theta characteristic} if $2[F] = [K_{\Gamma}]$. 
\end{defi}
\begin{rem}
There are $2^g$ theta characteristics on the graph $\Gamma$, as proved by Zharkov in \cite{Zha}. 
\end{rem} 
We give similar characterizations for theta characteristics on hyperelliptic graphs. Let $\Gamma$ a hyperelliptic graph of genus $g \geq 2$ with $r_1, \cdots, r_{g+1}$ ramification points, and the linear system $g^1_2 =|D|=  |2(r_i)| $. Without loss of generality we can set $D = 2r_1$.

We consider divisors of the form 
\begin{equation}\label{tc}  E = m D + r_{i_1} + \cdots + r_{i_{g-1-2m}}, \end{equation}
with $-1\leq m \leq \frac{g-1}{2}$ and the points $r_{i_j}$ are distinct. 

\begin{prop}
Let $E$ be a divisor as described in (\ref{tc}). The divisor class $[E]$ is a theta characteristics.\end{prop}
\begin{proof}
We need to prove that 
\[ 2E = 2m D + 2r_{i_1} + \cdots + 2r_{i_{g-1-2m}} \sim K. \]
By Proposition we know that $2(r_i) \sim 2(r_j)$ for every $i, j$. Therefore we obtain $2E \sim (g-1)D \sim K$ (for the last equivalence see Corollary 1 of \cite{Fac}). 
\end{proof}

\begin{prop} Let $[F]$ be a theta characteristic. Then $F$ is linearly equivalent to a unique divisor $E$ described in (\ref{tc}).
\end{prop}
\begin{proof} We start by showing that the divisors described in (\ref{tc}) are $r_1$-reduced. Let $X$ be a closed connected subset of $\Gamma \setminus \{r_1\}$.  We need to show that it contains a non-saturated boundary point. Consider the hyperelliptic involution $\iota$ on $\Gamma$ given by Theorem \ref{hyper} and let $T$ be the metric tree $\Gamma / \iota$. The image of $X$ through $\iota$ must have non-empty boundary in $T \setminus \{\iota(x)\}$. Therefore we can choose a point $x$ in the boundary of $X$, such that $\iota(x)$ is a boundary point of $\iota(X)$. From the inequality $\textrm{outdeg}_{\iota(X)}(\iota(x))\geq 1$, it follows that $\textrm{outdeg}_X(x) \geq 1$. Moreover if the point $x$ is a ramification point, then $\textrm{outdeg}_X(x) \geq 2$. This holds because the ramification points are fixed by the involution and every outgoing direction of $\iota(X)$ at $\iota(x)$ is covered by two tangent directions at $x$. In both cases we have $D(x) < \textrm{outdeg}_X(x)$, so we can conclude that $x$ is non-saturated. 
%

From this it follows that two divisors given by two different expressions in (\ref{tc}) are not linearly equivalent. 

\vspace{\baselineskip}
 We show that there are $2^g$ different divisors as in (\ref{tc}), proving in this way that all the theta characteristic need to be of that form, since they are $2^g$. The number of divisors is given by 
 \[
  \sum_{m=-1}^{\lfloor\frac{g-1}{2}\rfloor} \binom{g+1}{g-1-2m}.
   \]
   Remember that  $\sum_{i=0}^n \binom{n}{i} = 2^n$, and $\sum_{i=0}^{\lfloor\frac{n}{2}\rfloor} \binom{n}{2i} = 2^{n-1}$, we obtain  
 \[
   \sum_{m=-1}^{\lfloor\frac{g-1}{2}\rfloor} \binom{g+1}{g-1-2m} =  2^g.  
   \]

\end{proof}

\begin{prop} The rank of the divisors $E = m D + r_{i_1} + \cdots + r_{i_{g-1-2m}}$ is~$m$.
\end{prop}

\begin{proof} It is easy to see that $\textrm{rk}(E) \geq m$, since $E \geq mD = 2m(r_1)$. To show that the rank cannot be bigger is enough to choose a set $\{j_1, \dots, j_{m+1}\}$ such that the divisor 
\[F = \sum_{k=1}^{m+1} (r_{j_k})\]
has disjoint support from $E$.  
The divisor $E-F$ is not equivalent to any effective divisor, because 
\[ 
E - F \sim r_{i_1} + \cdots + r_{i_{g-1-2m}} + \sum_{k=1}^{m} (r_{j_k}) - (r_{j_{m+1}}),
\]
and the the divisor on the right-hand side is reduced with respect to $r_{j_{m+1}}$. 
\end{proof}
\end{section}

\begin{section}{Specialization of Theta characteristics} \label{5}
From Section \ref{specialization} we know that every minimal hyperelliptic graphs of genus at least two, with all the bridge edges contracted has a morphism of degree two to a tree that can be lifted to a morphism of a hyperelliptic curve $X$ to the projective line. The curve has $2g+2$ ramification points which specialize in pairs to the $g+1$ ramification points on the graph. Given a theta characteristic on the graph, we want to understand how many theta characteristics on the curve specialize to it. 

\begin{thm} \label{counting}
Let $\Gamma$ be a hyperelliptic graphs  of genus $g\geq 2$ with all the bridge edges contracted. Let $X$ be the hyperelliptic curve constructed by lifting the harmonic morphism $\varphi: \Gamma \rightarrow T$. There are $2^g$ even theta characteristic on $X$ specializing to the unique non-effective theta characteristic on~$\Gamma$. For every effective theta characteristic on $\Gamma$ there are $2^{g-1}$ odd and $2^{g-1}$ even theta characteristics on $X$ specializing to it. 
 
\end{thm}

\begin{proof}
The only non-effective theta characteristic on the graph is given by the divisor 
\[
 E = -(r_1) + (r_2) + \cdots + (r_{g+1}). 
\]
For every point $r_i$ there are two ramification points in the curve specializing to it, so there are $2^{g+1}$ divisors specializing to it. By the linear equivalence in (\ref{equivalenza}) this number needs to be divided by $2$. Therefore there are $2^g$ theta characteristics specializing to $E$. These theta characteristics are even, since $m=-1$. 

\vspace{\baselineskip}
Now let $E$ be an effective theta characteristic. We write 
\[
 E = 2m (r_1)  + (r_{i_1}) + \cdots + (r_{i_{g-1-2m}}), \ \textrm{with} \ m \not = -1. 
\]
Again for each ramification point $r_{i_1}, r_{i_2}, \cdots, r_{i_{g-1-2m}}$ in the support, there are two ramification points on the curve that can specialize to it. This gives us $2^{g-1-2m}$ possible divisors specializing to $(r_{i_1}) + \cdots + (r_{i_{g-1-2m}})$

Now we focus on the divisors that specialize to $2m(r_i)$. A divisor equivalent to $2m(r_1)$ can be obtained by the specialization of $m \,g^1_2$ on the curve, but, for example, also from the specialization of $(m-1)\,g^1_2$ and a pair of ramification points that specialize to the same point $r_j$ on the graph, since $2(r_j) \sim g^1_2$. 

We have that the theta characteristics on $X$ specializing to $E$ are of the form 
\[
E' = jg^1_2 + \sum_{h=1}^i T_h + (R_{i_1}) + \cdots + (R_{i_{g-1-2m}})
\]
where $-1 \leq j \leq m$, and $j+i= m$, and with the notation $T_h$ we indicate a pair of ramification points on the curve specializing to the same ramification point on the graph not in the support of $E$, and with $R_{i_j}$ a point specializing to $r_{i_j}$. For each $j$ there are $ \binom{2m+2}{i}$ divisors specializing to~$jg^1_2 + \sum_{h=1}^i T_h$. In the case $i=m+1$ the number of divisors needs again to be divided by $2$ because of the linear equivalence in (\ref{equivalenza}).  

Putting these computations together we obtain that given a theta characteristic on a graph, the number of theta characteristics specializing to it is  
\begin{equation*}\begin{split}
 2^{g-1-2m}\left(\sum_{i=0}^{m} \binom{2m+2}{i} + \frac{1}{2}\binom{2m+2}{m+1}\right) &= 2^{g-1-2m}\left(\frac{1}{2}\sum_{i=0}^{2m+2} \binom{2m+2}{i} \right) \\ &= 2^{g-1-2m} \cdot 2^{2m+2-1} = 2^g. 
\end{split}\end{equation*}

\bigskip

 We know that $h^0 (X,\mathcal{O}_X(E')) = j+1$, with $E'$ the theta characteristic above. 
Using again the fact that $\sum_{i=0}^{\lfloor \frac{n}{2}\rfloor} \binom{n}{2i} = 2^{n-1}$, we can see that the number of even theta characteristics specializing to it is $2^{g-1}$ and the number of odd theta characteristics specializing to it is $2^{g-1}$. 
\end{proof}
\end{section}

\bibliographystyle{amsalpha}
\bibliography{bibliografia}
\end{document}